\documentclass[12pt,reqno]{amsart}
\usepackage{fullpage}
\usepackage{times}
\usepackage{amsmath,amssymb,amsthm,url}
\usepackage[utf8]{inputenc}
\usepackage[english]{babel}
\usepackage{bbm}
\usepackage{enumerate}
\usepackage{bm}
\usepackage{graphicx}
\usepackage{mathrsfs}
\usepackage[colorlinks=true, pdfstartview=FitH, linkcolor=blue, citecolor=blue, urlcolor=blue]{hyperref}
\usepackage{tikz-cd}
\usepackage{tabstackengine}
\stackMath

\newtheorem{thm}{Theorem}[section]
\newtheorem{lem}[thm]{Lemma}
\newtheorem{prop}[thm]{Proposition}
\newtheorem{cor}[thm]{Corollary}

\theoremstyle{definition}
\newtheorem{defn}[thm]{Definition}

\newtheorem{rem}[thm]{Remark}
\newtheorem{ex}[thm]{Example}

\newcommand{\F}{\mathbb{F}}

\title{Additive decompositions of large multiplicative subgroups in finite fields}
\author{Chi Hoi Yip}
\address{Department of Mathematics \\ University of British Columbia \\ Vancouver  V6T 1Z2 \\ Canada}
\email{kyleyip@math.ubc.ca}
\keywords{sumset, additive decomposition, multiplicative subgroup, finite field}
\subjclass[2020]{11B13, 11B30, 11P70, 11T06}
\begin{document}

\maketitle

\begin{abstract}
We show that a large multiplicative subgroup of a finite field $\mathbb{F}_q$ cannot be decomposed into $A+A$ or $A+B+C$ nontrivially. We also find new families of multiplicative subgroups that cannot be decomposed as the sum of two sets nontrivially. In particular, our results extensively generalize the results of S\'{a}rk\"{o}zy and Shkredov on the additive decomposition of the set of quadratic residues modulo a prime. 
\end{abstract}

\section{Introduction}

Throughout the paper, let $p$ be a prime and $q$ a power of $p$. Let $\F_q$ be the finite field with $q$ elements. We always assume that $\F_q$ has characteristic $p$. Let $d \mid (q-1)$ such that $d>1$, we denote $S_d=\{x^d: x \in \F_q^*\}$ to be the subgroup of $\F_q^*$ with order $\frac{q-1}{d}$.  

A celebrated conjecture of S\'{a}rk\"{o}zy \cite{S12} asserts that $S_2$ cannot be written as $S_2=A+B$, where $A,B \subset \F_q$ and $|A|, |B| \geq 2$, provided that $q$ is a sufficiently large prime. Recall that for two sets $A,B \subset \F_q$, their {\em sum} is $A+B=\{a+b: a \in A, b \in B\}$. It is necessary to assume that $|A|,|B| \geq 2$ in the conjecture, otherwise there are many {\em trivial} additive decompositions of $S_2$ into the sum of two sets. The motivation of the conjecture is clear: sumsets have rich additive structure, while multiplicative subgroups do not possess too much additive structure. It is natural to consider the analogue of this conjecture over all finite fields and all multiplicative subgroups, namely, given a proper multiplicative subgroup $G$ of $\F_q$, can we find a nontrivial additive decomposition of $G$ as the sum of two sets $A, B \subset \F_q$? This natural extension was first studied by Shparlinski \cite{S13}, shortly after S\'{a}rk\"{o}zy's seminal paper \cite{S12} was published. It is likely that the answer remains negative provided that the index of the subgroup $G$ is fixed and $q$ is sufficiently large, and this can be regarded as \emph{the generalized S\'{a}rk\"{o}zy's conjecture}. S\'{a}rk\"{o}zy's conjecture, as well as the generalized S\'{a}rk\"{o}zy's conjecture, have attracted lots of attention in the last decade; see for example \cite{S13, S14, BKS15, S16, S20, HP, CY21, CX22, WS23} (listed in chronological order). 

Recently, Hanson and Petridis \cite{HP} showed that if $q$ is a prime and $S_d=A+B$ with $A, B \subset \F_q$ and $|A|,|B| \geq 2$, then $|A||B|=|S_d|$, that is, all sums $a+b$ must be distinct. This is a breakthrough on S\'{a}rk\"{o}zy's conjecture, and in particular, it implies that S\'{a}rk\"{o}zy's conjecture holds for almost all primes \cite[Corollary 1.4]{HP}. We also refer to the recent work by Chen and Xi \cite{CX22} on the best-known lower and upper bounds on $|A|$ and $|B|$ provided that $A+B=S_2$ and $A, B \subset \F_p$.

In this paper, we make progress on the generalized S\'{a}rk\"{o}zy's conjecture for large multiplicative subgroups in finite fields. In particular, we show that large multiplicative subgroups cannot be written as $A+A$ or $A+B+C$ nontrivially, extending and refining existing results in the literature significantly. We also obtain new families of $(d,q)$ such that $S_d \subset \F_q$ admits no nontrivial additive decomposition, as well as new structural information on $A$ and $B$ if $A+B=S_d$ (Proposition~\ref{prop:structure}). 

Our first result extends \cite[Theorem 1.2]{HP}, which is the main result in \cite{HP} by Hanson and Petridis, to all finite fields, with an extra assumption on the non-vanishing of a binomial coefficient.

\begin{thm}\label{thm:main}
Let $q$ be a power of a prime $p$ and let $d \mid (q-1)$ such that $d>1$. If $A,B \subset \F_q$ such that $A+B\subset S_d \cup \{0\}$ and $
\binom{|A|-1+\frac{q-1}{d}}{\frac{q-1}{d}}\not \equiv 0 \pmod p,$
then 
$$|A||B|\leq \frac{q-1}{d}+|A \cap (-B)|.$$
\end{thm}

We remark that in general the condition on the binomial coefficient cannot be dropped; see Example~\ref{ex:ex1}. Note that when $q$ is a prime, it is easy to see that the condition on the binomial coefficient is automatically satisfied and thus Theorem~\ref{thm:main} recovers \cite[Theorem 1.2]{HP}.

With a bit extra work, Theorem~\ref{thm:main} implies the following corollary (see also Proposition~\ref{prop:structure}), which generalizes \cite[Corollary 1.3]{HP} and makes new progress towards the generalized S\'{a}rk\"{o}zy's conjecture.

\begin{cor}\label{cor:a+b}
Let $d \mid (q-1)$ such that $\frac{q-1}{d}\leq \frac{2p}{3}$. If $A,B \subset \F_q$ such that $A+B=S_d$, then $|A||B|\leq \frac{q-1}{d}$; moreover, if $A+B=S_d$, then $|A||B|=\frac{q-1}{d}$ and all sums $a+b$ are distinct. In particular, if $\frac{q-1}{d}$ is a prime, then there is no nontrivial additive decomposition of $S_d$ into the sum of two sets.
\end{cor}

We remark that Corollary~\ref{cor:a+b} does not extend to all pairs $(d,q)$ with $d \geq 2$ and $d \mid (q-1)$; see Example~\ref{ex:ex1}. Nevertheless, if $A+B=S_d$ and $|A||B|>|S_d|$, we can still say something about the structure of $A$ and $B$; see Proposition~\ref{prop:structure} and Remark~\ref{rem:structure}.

Next, we proceed with the discussion on the possibility of $S_d=A+A$. Before stating our contributions, we introduce the following two definitions for convenience. Let $q=p^n$ and let $d \mid (q-1)$ with $2 \leq d<q-1$. Write the base-$p$ expansion of $\frac{q-1}{d}$ as $$\frac{q-1}{d}=\sum_{j=0}^{n} e_jp^j,$$ where $0 \leq e_j \leq p-1$. For a positive real number $\delta$, we say a pair $(d,q)$ is \emph{$\delta$-good} if $$e_j \leq \lfloor (1-\delta)(p-1) \rfloor$$ for each $0 \leq j<(n+1)/2$. Furthermore, we say a pair $(d,q)$ is $\emph{good}$ if one of the following conditions holds:
\begin{itemize}
    \item $\frac{q-1}{d}\leq \frac{2p}{3}$;
    \item $e_j \leq \frac{p-1}{2}$ for each $0 \leq j<n/2$;
    \item $n=2r+1$ with $r \geq 1$, $d\leq 2p-2$, and $e_r\leq p-1-\lceil \sqrt{p} \rceil/2$; 
    \item $n=2r$, $d \leq 2p^2$, and $e_{r-1} \leq \frac{p-3}{2}$.
    \end{itemize}
Note that if a pair $(d,q)$ is $\frac{1}{2}$-good, then it is good. 

Shkredov \cite{S14} showed that $S_2$ cannot be written as $A+A$ when $q=p$ is a prime such that $p>3$. However, as remarked by Shparlinski \cite{S13} (see also Remark~\ref{rmk:A+A}), his method does not seem to extend to other subgroups. Instead, Theorem~\ref{thm:main}, together with other tools and observations, allows us to extend Shkredov's result to other subgroups. 

\begin{thm}\label{thm:A+A}
Let $d \mid (q-1)$ with $2 \leq d<q-1$. If the pair $(d,q)$ is good, then there is no additive decomposition $S_d=A+A$ for any $A \subset \F_q$.
\end{thm}

We remark that when $q$ is a proper prime power, we do need some extra assumptions on the pair $(d,q)$ for Theorem~\ref{thm:A+A} to hold; see Example~\ref{ex:ex1} for a counterexample. In certain cases, it is easy to verify that a pair $(d,q)$ is good without actually expanding $\frac{q-1}{d}$ in base-$p$. The following corollary describes a few such cases and is thus an immediate consequence of Theorem~\ref{thm:A+A}. 

\begin{cor}\label{cor:A+A}
Let $q=p^n$ and $d \mid (q-1)$ with $2 \leq d<q-1$. Let $k$ be the order of $p$ modulo $d$. Then there is no additive decomposition $S_d=A+A$ for any $A \subset \F_q$, provided that one of the following conditions holds:
\begin{itemize}
    \item $d \mid (p-1)$ (in particular, if $d=2$ or $q$ is a prime);
    \item $\frac{p^k-1}{d}\leq \frac{p-1}{2}$;
    \item $2k \mid n$ and $d \leq 2p^2$.
\end{itemize}
\end{cor}

Finally, we turn to the discussion of the {\em ternary decomposition}, that is, decomposing $S_d$ into the sum of three sets. In his paper, S\'{a}rk\"{o}zy \cite{S12} confirmed the ternary version of his conjecture. Recently, Chen and Yan \cite{CY21} showed that for any prime $p$, there is no additive decomposition $A+B+C=S_2$ with $|A|,|B|,|C| \geq 2$, refining S\'{a}rk\"{o}zy's result \cite{S12}. Wu and She \cite{WS23} showed that for a prime $p> 184291$ with $p \equiv 1 \pmod 3$, there is no nontrivial additive decomposition of $S_3$. More generally, Wu, Wei, and Li \cite[Theorem 1.2]{WWL24} showed that if $d \geq 2$ and $p \equiv 1 \pmod d$ is a prime such that $p>217^2d^4$, then $S_d$ has no nontrivial ternary additive decomposition.  We refer to \cite{GS15, GS17} for a general discussion on (ternary)-irreducible subsets of $\F_p$. On the other hand, Shkredov \cite{S20} showed that any small multiplicative subgroup is not a sumset in the sense that if $G$ is a multiplicative subgroup of $\F_p$ with $1 \ll_{\epsilon} |G| \leq p^{2/3-\epsilon}$, then there is no additive decomposition $G=A+B$ with $A, B \subset \F_p$ and $|A|,|B| \geq 2$. 

Our next result shows that as long as $G$ is a proper multiplicative subgroup of some prime field $\F_p$ with $|G| \gg 1$, then $G$ cannot be written as $A+B+C$ nontrivially. In particular, this confirms the ternary version of the generalized S\'{a}rk\"{o}zy's conjecture over prime fields in a stronger form, and refines the recent result of Wu, Wei, and Li \cite{WWL24}.

\begin{thm}\label{thm:prime3}
There exists a constant $M>0$, such that whenever $G$ is a proper multiplicative subgroup of $\F_p$ with $|G|>M$ and $p$ a prime, there is no additive decomposition $G=A+B+C$ with $A,B,C \subset \F_p$ and $|A|,|B|,|C| \geq 2$.
\end{thm}

Theorem~\ref{thm:prime3} fails to extend to general finite fields $\F_q$; a simple counterexample can be found in Example~\ref{ex:ex2}. Instead, we show that a large multiplicative subgroup $S_d$ cannot be written as $A+B+C$ nontrivially. 

\begin{thm}\label{thm:dq}
Let $\epsilon>0$. There is a constant $Q=Q(\epsilon)$, such that for each prime power $q>Q$ and a divisor $d$ of $q-1$ with $2 \leq d \leq q^{1/10-\epsilon}$, there is no additive decomposition $S_d=A+B+C$ with $A,B,C \subset \F_q$ and $|A|,|B|,|C| \geq 2$. 
\end{thm}

In particular, Theorem~\ref{thm:dq} implies the following corollary immediately, which confirms the ternary version of the generalized S\'{a}rk\"{o}zy's conjecture.

\begin{cor}\label{cor:dfixed}
There is an absolute constant $M$, such that for each pair $(d,q)$ such that $q \equiv 1 \pmod d$ and $q>Md^{11}$, there is no additive decomposition $S_d=A+B+C$ with $A,B,C \subset \F_q$ and $|A|,|B|,|C| \geq 2$. In particular, if $d$ is fixed, and $q \equiv 1 \pmod d$ is sufficiently large, then there is no nontrivial ternary decomposition of $S_d$.
\end{cor}

Assuming that the base-$p$ representation of $\frac{q-1}{d}$ behaves ``nicely", we can further improve the range of $d$ from $q^{1/10}$ in Theorem~\ref{thm:dq} to roughly $q^{1/4}$.

\begin{thm}\label{thm:deltagood}
Let $\epsilon>0$. There are constants $Q=Q(\epsilon)$ and $P=P(\epsilon)$, such that for each prime power $q>Q$ with $p>P$, and a divisor $d$ of $q-1$ such that the pair $(d,q)$ is $\delta$-good with $\delta>0$ and $2 \leq d \leq q^{1/4-\epsilon} \delta^{3/2}$, there is no nontrivial ternary decomposition of $S_d$.
\end{thm}

\textbf{Notation.} We follow the Vinogradov notation $\ll$. We write $X \ll Y$ if there is an absolute constant $C>0$ so that $|X| \leq CY$.

\textbf{Structure of the paper.} 
In Section~\ref{sec:prelim}, we provide additional background and prove some preliminary results. In Section~\ref{sec:Stepanov}, we prove Theorem~\ref{thm:main} and Corollary~\ref{cor:a+b}. In Section~\ref{sec:A+A}, we prove Theorem~\ref{thm:A+A}. In Section~\ref{sec:A+B+C}, we prove Theorem~\ref{thm:prime3},  Theorem~\ref{thm:dq}, and Theorem~\ref{thm:deltagood}. 

\section{Preliminaries}\label{sec:prelim}

\subsection{Estimates of the size of $A$ and $B$ for $A+B=S_d$} In this section, we collect a few known results on the size of $A$ and $B$ for a given additive decomposition $S_d=A+B$.

\begin{lem}\label{lem:ub}
Let $A,B \subset \F_q$. If $A+B\subset S_d$, then $|A||B|<q$. 
\end{lem}
\begin{proof}
Let $\chi$ be a multiplicative character of $\F_q$ with order $d$. The following double character sum estimate is well-known (see for example \cite[Theorem 2.6]{Y22}):
 \begin{equation}\label{eq:double}
 \bigg|\sum_{a\in A,\, b\in B}\chi(a+b)\bigg|  \leq \sqrt{q|A||B|}\bigg(1-\frac{|A|}{q}\bigg)^{1/2}\bigg(1-\frac{|B|}{q}\bigg)^{1/2}.
 \end{equation}
Since $A+B\subset S_d $, we have $\chi(a+b)=1$ for all $a \in A$, $b \in B$. It follows that
$$
|A||B| \leq \sqrt{q|A||B|}\bigg(1-\frac{|A|}{q}\bigg)^{1/2}\bigg(1-\frac{|B|}{q}\bigg)^{1/2}<\sqrt{q|A||B|}
$$
and thus $|A||B|<q$. 
\end{proof}

Note that the right-hand side of inequality~\eqref{eq:double} is symmetric in $|A|$ and $|B|$. Based on the application of a bound of Karatsuba~\cite{K91} (more precisely, see \cite[Lemma 2.2]{S13}) on double character sums, that is, an ``asymmetric version" of inequality~\eqref{eq:double}, Shparlinski \cite{S13} proved the following remarkable theorem.

\begin{thm}[Shparlinski {\cite[Theorem 7.1]{S13}}]\label{thm:lball}
If $G$ is a proper multiplicative subgroup of $\F_p$ such that $G=A+B$ for some $A,B \subset \F_p$ with $|A|,|B| \geq 2$, then 
$$
|G|^{1/2 + o(1)}= \min\{|A|, |B|\} \leq \max \{|A|, |B|\}=|G|^{1/2 + o(1)}
$$
as $|G| \to \infty$.
\end{thm}  

As remarked by Shparlinski~\cite{S13}, one of the main ingredients to prove Theorem~\ref{thm:lball}, namely \cite[Lemma 3.1]{S13} on the size of the intersection of shifted subgroups, has been established only for prime fields. Indeed, Theorem~\ref{thm:lball} fails to extend to finite fields in general; see Example~\ref{ex:ex1} and Example~\ref{ex:ex2}. Nevertheless, for a large multiplicative subgroup of a general finite field, we have the following weaker estimate, which was implicitly discussed in \cite[Section 6]{S13}:

\begin{lem}[Shparlinski \cite{S13}]\label{lem:maxmin}
Let $\epsilon>0$. Let $d \mid (q-1)$ such that $d \leq q^{1/4-\epsilon}$.  If $A+B=S_d$ for some $A,B \subset \F_q$ with $|A|, |B| \geq 2$, then
    $$
    \frac{\sqrt{q}}{d} \ll \min \{|A|,|B|\} \leq \max \{|A|,|B|\} \ll q^{1/2}.
    $$
\end{lem}
\begin{proof}
Since $d \leq q^{1/4-\epsilon}$, we have $|S_d| \gg q^{3/4+\epsilon}$. Thus, \cite[Theorem 6.1]{S13} implies that 
\begin{equation}\label{eq:max}
\max \{|A|,|B|\} \ll q^{1/2}.    
\end{equation}
On the other hand, note that $A+B=S_d$ implies $|A||B|\geq |S_d| \gg \frac{q}{d}$. Thus, it follows that 
 $$
    \min \{|A|,|B|\}=\frac{|A||B|}{ \max \{|A|,|B|\}} \gg \frac{\sqrt{q}}{d}.
    $$
\end{proof}

\subsection{Hyper-derivatives}
The proof of Theorem~\ref{thm:main} replies on computing derivatives of a polynomial over $\F_q$. Note that the $p$-th derivative of a polynomial over $\F_q$ is trivially $0$ due to its characteristic. To overcome the characteristic issue, we need to work on hyper-derivatives (also known as Hasse derivatives) instead. We recall a few basic properties of hyper-derivative; a general discussion can be found in \cite[Section 6.4]{LN97}. 

\begin{defn}
Let $c_0,c_1, \ldots c_d \in \F_q$. If $n$ is a non-negative integer, then the $n$-th hyper-derivative of $f(x)=\sum_{j=0}^d c_j x^j$ is
$$
E^{(n)}(f) =\sum_{j=0}^d \binom{j}{n} c_j x^{j-n},
$$
where we follow the standard convention that $\binom{j}{n}=0$ for $j<n$, so that the $n$-th hyper-derivative is a polynomial.
\end{defn}

\begin{lem}[{\cite[Corollary 6.48]{LN97}}]\label{lem:differentiate}
Let $n,d$ be positive integers. If $c \in \F_q$, then 
$$E^{(n)}\big((x+c)^d\big)=\binom{d}{n} (x+c)^{d-n}.$$
\end{lem}

\begin{lem}[{\cite[Lemma 6.51]{LN97}}]\label{lem:multiplicity}
Let $f$ be a nonzero polynomial in $\F_q[x]$. If $c$ is a root of $E^{(k)}(f)$ for $k=0,1,\ldots, m-1$, then $c$ is a root of multiplicity at least $m$. 
\end{lem}

\subsection{Tools from additive combinatorics}
In this section, we list two useful results from additive combinatorics.

The following theorem is a generalization of the classical Cauchy-Davenport theorem.
\begin{thm}[K\'{a}rolyi \cite{K05}]\label{thm:CD}
Let $G$ be a finite group with group operation $+$, and let $A, B \subset G$ be non-empty subsets. Further, let $\rho(G)$ denote the minimum of the orders of nontrivial subgroups of $G$. If $\rho(G)\geq |A|+|B|-1$, then $|A+B|\geq |A|+|B|-1$.   
\end{thm}

\begin{cor}\label{cor:CD}
Let $A, B \subset \F_q$ be non-empty subsets. Then $|A+B|\geq \min \{p, |A|+|B|-1\}$.
\end{cor}
\begin{proof}
The minimum size of a nontrivial (additive) subgroup of $\F_q$ is $p$. Thus, the corollary follows from Theorem~\ref{thm:CD}.    
\end{proof}

The following lemma turns out to be useful for studying ternary decompositions. 

\begin{lem}[Ruzsa \cite{R07}]\label{lem:addcomb}
Let $A, B, C$ be nonempty subsets of $\F_q$. Then $$|A+B+C|^2 \leq |A+B||B+C||C+A|.$$    
\end{lem}
\begin{proof}
This is a special case of \cite[Theorem 5.1]{R07} due to Ruzsa (see also a generalization in \cite[Theorem 1.2]{GMR10} by Gyarmati, Matolcsi, and Ruzsa). 
\end{proof}

\section{Proof of Theorem~\ref{thm:main}}\label{sec:Stepanov}

The proof of Theorem~\ref{thm:main} is based on Stepanov's method. The key idea is to construct a low degree \emph{nonzero} polynomial that vanishes on each element of $B$ with high multiplicity.  

\begin{proof}[Proof of Theorem~\ref{thm:main}]
If $|A|=1$, the result is immediate. Next we assume that $|A|\geq 2$. Let $r=|A \cap (-B)|$. Let $A=\{a_1,a_2,\ldots, a_n\}$ and $B=\{b_1,b_2, \ldots, b_m\}$ such that $b_{r+1}, \ldots, b_m \notin (-A)$. Since $A+B \subset S_d \cup \{0\}$, we have $$(a_i+b_j)^{\frac{q-1}{d}+1}=a_i+b_j$$ for each $1 \leq i \leq n$ and $1 \leq j \leq m$. This simple observation will be used repeatedly in the following computation.

Let $c_1,c_2,...,c_n$ be the unique solution of the following system of equations:
\begin{equation} \label{system} 
\left\{
\TABbinary\tabbedCenterstack[l]{
\sum_{i=1}^n c_i a_i^j=0,  \quad 0 \leq j \leq n-2\\\\
\sum_{i=1}^n c_i a_i^{n-1}=1
}\right.    
\end{equation}
This is justified by the invertibility of the coefficient matrix of the system (a Vandermonde matrix).

Consider the following auxiliary polynomial 
$$
f(x)=-1+\sum_{i=1}^n c_i (x+a_i)^{n-1+\frac{q-1}{d}}\in \F_q[x].
$$
We claim that the degree of $f$ is $\frac{q-1}{d}$. Indeed, for each $0 \leq j \leq n-1$, the coefficient of $x^{n-1+\frac{q-1}{d}-j}$ in $f(x)$ is 
$$
\binom{n-1+\frac{q-1}{d}}{j} \cdot \sum_{i=1}^n c_i a_i^{j}.
$$
Thus, system~\eqref{system} implies that the coefficient of $x^{n-1+\frac{q-1}{d}-j}$
is $0$ for $j=0,1, \ldots, n-2$, and the coefficient of $x^{\frac{q-1}{d}}$
is $$\binom{n-1+\frac{q-1}{d}}{n-1}= \binom{n-1+\frac{q-1}{d}}{\frac{q-1}{d}}\neq 0$$ by the assumption. 

Next, we compute the hyper-derivatives of $f$ on $B$. For each $1\leq j \leq m$, system~\eqref{system} implies that
\begin{align*}
E^{(0)} f (b_j)
&= f (b_j)\\
&=  -1+\sum_{i=1}^n c_i (b_j+a_i)^{n-1+\frac{q-1}{d}}  \\
&=  -1+\sum_{i=1}^n c_i (b_j+a_i)^{n-1}   \\
&=  -1+\sum_{\ell=0}^{n-1} \binom{n-1}{\ell} \bigg(\sum_{i=1}^n c_i a_i^\ell\bigg)b_j^{n-1-\ell}    \\
&= -1+\binom{n-1}{n-1} \cdot \sum_{i=1}^n c_i a_i^{n-1}   
=0.
\end{align*}

For each $1\leq j \leq m$ and $1 \leq k \leq n-2$, Lemma~\ref{lem:differentiate} implies that 
\begin{align*}
E^{(k)} f (b_j)
&= \binom{n-1+\frac{q-1}{d}}{k}  \sum_{i=1}^n c_i (b_j+a_i)^{n-1+\frac{q-1}{d}-k} \\
&= \binom{n-1+\frac{q-1}{d}}{k}  \sum_{i=1}^n c_i (b_j+a_i)^{n-1-k}   \\
&= \binom{n-1+\frac{q-1}{d}}{k}  \sum_{\ell=0}^{n-1-k} \binom{n-1-k}{\ell} \bigg(\sum_{i=1}^n c_i a_i^\ell\bigg)b_j^{n-1-k-\ell}   
=0,
\end{align*}
where we again use the assumptions in system~\eqref{system}.

For each $r+1\leq j \leq m$, by the assumption, $b_j \notin (-A)$, that is, $b_j+a_i \neq 0$ for each $1 \leq i \leq n$. Thus, by Lemma~\ref{lem:differentiate}, for each $r+1\leq j \leq m$, we additionally have
\begin{align*}
E^{(n-1)} f (b_j)
&= \binom{n-1+\frac{q-1}{d}}{n-1}  \sum_{i=1}^n c_i (b_j+a_i)^{\frac{q-1}{d}} \\
&= \binom{n-1+\frac{q-1}{d}}{n-1}  \sum_{i=1}^n c_i  =0.
\end{align*}

Given Lemma \ref{lem:multiplicity}, we conclude that each of $b_1,b_2, \ldots b_r$ is a root of $f$ with multiplicity at least $n-1$, and each of $b_{r+1},b_{r+2}, \ldots b_m$ is a root of $f$ with multiplicity at least $n$. 
Therefore
$$
r(n-1)+(m-r)n= mn-r \leq \operatorname{deg}f =\frac{q-1}{d}.
$$
\end{proof}

\begin{rem}\label{rem:Paley}
The proof of Theorem~\ref{thm:main} is inspired by the arguments in \cite{HP}, as well as the arguments in \cite{Yip1, Y21} related to the clique number of (generalized) Paley graphs. Recall that if $q \equiv 1 \pmod {2d}$, then a subset $C \subset \F_q$ in the $d$-Paley graph over $\F_q$ is a clique if and only if $C-C \subset S_d \cup \{0\}$. Let $C$ be a maximum clique in the $d$-Paley graph over $\F_q$; by taking $A$ to be a subset of $C$ with the desired condition on the binomial coefficient and $B=-C$ in the Theorem~\ref{thm:main}, it recovers \cite[Theorem 1.6]{Yip1} and \cite[Theorem 5.8]{Y21} immediately, which are key ingredients in establishing the best-known upper bound on the clique number of Paley graphs \cite{Yip1} and generalized Paley graphs \cite{Y21} over finite fields. In particular, the estimate on the clique number of cubic Paley graphs by the author~\cite[Theorem 1.8]{Y21} has already improved the upper bound on the size of $A \subset \F_p$ such that $A-A=S_3 \cup \{0\}$ proved in the recent work of Wu and She \cite[Theorem 1.4]{WS23} (in fact the result by the author works for a general finite field $\F_q$), perhaps because this connection has been sometimes overlooked. On the other hand, although the work of Hanson and Petridis~\cite{HP} has received lots of attention in the study of Paley graphs, surprisingly, their work did not receive much attention in the study of additive decompositions.
\end{rem}

The following corollary is a simple consequence of Theorem~\ref{thm:main}.

\begin{cor}\label{cor:cora+b}
If $A,B \subset \F_q$ such that $A+B\subset S_d$ and $
\binom{|A|-1+\frac{q-1}{d}}{\frac{q-1}{d}}\not \equiv 0 \pmod p,$
then 
$|A||B|\leq \frac{q-1}{d}$. 
\end{cor}
\begin{proof}
Since $0 \notin A+B$, it follows that $A \cap (-B)$ is empty. The corollary follows from Theorem~\ref{thm:main} immediately.   
\end{proof}

As a special case, we deduce Corollary~\ref{cor:a+b} as a partial progress towards the generalized S\'{a}rk\"{o}zy's conjecture.

\begin{proof}[Proof of Corollary~\ref{cor:a+b}]
Let $A,B \subset \F_q$ such that $A+B \subset S_d$. Without loss of generality, we may assume that $|A| \leq |B|$. Since $$|A+B|\leq |S_d|=\frac{q-1}{d}<p,$$ Corollary~\ref{cor:CD} implies that 
$$
2(|A|-1) \leq |A|+|B|-2 \leq |A+B|-1=\frac{q-1}{d}-1.
$$
It follows that
$$
\frac{q-1}{d}\leq |A|-1+\frac{q-1}{d}<\frac{3(q-1)}{2d}\leq p,
$$
and thus $
\binom{|A|-1+\frac{q-1}{d}}{\frac{q-1}{d}}\not \equiv 0 \pmod p.$ 
Corollary~\ref{cor:cora+b} then implies that $|A||B| \leq \frac{q-1}{d}$. If we further assume that $A+B=S_d$, then $$|A||B| \geq |A+B|=|S_d|=\frac{q-1}{d},$$ which forces that $|A||B|=\frac{q-1}{d}$ and thus all sums $a+b$ are distinct.
\end{proof}

Next, we take a closer look at the proof of Theorem~\ref{thm:main} and make new observations.

\begin{rem}
Suppose $A,B \subset \F_q$ such that $A+B=S_d$ and $
\binom{|A|-1+\frac{q-1}{d}}{\frac{q-1}{d}}\equiv 0 \pmod p$. Then Theorem~\ref{thm:main} does not apply. However, we can still say something nontrivial from the proof of Theorem~\ref{thm:main}. In the proof of Theorem~\ref{thm:main}, the assumption on the binomial coefficient guarantees that the polynomial $f$ is nonzero, so that we can apply Lemma~\ref{lem:multiplicity} to obtain an upper bound on $|A||B|$ based on the degree of $f$. It is clear that without the condition on the binomial coefficient, we can still conclude the same upper bound if the polynomial $f$ is nonzero.

Indeed, we can say something stronger. We follow the same notations as in the proof of Theorem~\ref{thm:main}. Since $A+B=S_d$, we have $r=0$. The same computations would show that each element in $B$ is a root of $f$ with multiplicity at least $n=|A|$. However, since $
\binom{|A|-1+\frac{q-1}{d}}{\frac{q-1}{d}}\equiv 0 \pmod p$, the degree of $f$ is strictly less than $\frac{q-1}{d}$. If $f$ is a nonzero polynomial, then Lemma~\ref{lem:multiplicity} implies that
$$
\frac{q-1}{d}=|S_d|=|A+B| \leq |A||B|\leq \deg f <\frac{q-1}{d},
$$
a contradiction. Therefore, the polynomial $f$ must be zero.
\end{rem}

We summarize the above discussions into the following proposition. Roughly speaking, it predicts that any supposed additive decompositions of $S_d$ must have a very rigid structure, providing some evidence for the generalized S\'{a}rk\"{o}zy's conjecture.

\begin{prop}\label{prop:structure}
Let $d \mid (q-1)$ such that $d>1$. If $A,B \subset \F_q$ such that $A+B=S_d$ with $|A|, |B| \geq 2$, then one of the following two situations happens:
\begin{itemize}
    \item $|A||B|=|S_d|$, that is, all sums $a+b$ are distinct;
    \item $|A||B|>|S_d|$. In this case, $|A|$ and $|B|$ must satisfy  
    $$
    \binom{|A|-1+\frac{q-1}{d}}{\frac{q-1}{d}} \equiv 0 \pmod p, \quad \binom{|B|-1+\frac{q-1}{d}}{\frac{q-1}{d}} \equiv 0 \pmod p.
    $$
Moreover, if we write $A=\{a_1,a_2, \ldots, a_n\}$ and $B=\{b_1, b_2, \ldots, b_m\}$, then 
$$
\binom{n-1+\frac{q-1}{d}}{j} \cdot \sum_{i=1}^n c_i a_i^{j}=0, \quad \binom{m-1+\frac{q-1}{d}}{\ell} \cdot \sum_{k=1}^m d_k b_k^{\ell}=0
$$
for all $0 \leq j<n-1+\frac{q-1}{d}$ and $0 \leq \ell<m-1+\frac{q-1}{d}$, and
$$
\sum_{i=1}^n c_i a_i^{n-1+\frac{q-1}{d}}=1, \quad \sum_{k=1}^m d_k b_k^{m-1+\frac{q-1}{d}}=1,
$$
where $c_i$'s and $d_k$'s are uniquely determined by the following systems of linear equations:
\begin{equation} \label{system111} 
\left\{
\TABbinary\tabbedCenterstack[l]{
\sum_{i=1}^n c_i a_i^j=0,  \quad 0 \leq j \leq n-2\\\\
\sum_{i=1}^n c_i a_i^{n-1}=1\\\\
\sum_{k=1}^m d_k b_k^\ell=0,  \quad 0 \leq \ell \leq m-2\\\\
\sum_{k=1}^m d_k b_k^{m-1}=1
}\right.    
\end{equation}
\end{itemize}

\end{prop}

\begin{rem}\label{rem:structure}
Following the notations used in Proposition~\ref{prop:structure}, suppose that $A+B=S_d$ with $|A||B|>|S_d|$. Proposition~\ref{prop:structure} implies that there are lots of generalized Vandermonde matrices associated with $A$ that are singular. In particular, for each $n \leq j_0<n-1+\frac{q-1}{d}$ such that $\binom{n-1+\frac{q-1}{d}}{j_0} \not \equiv 0 \pmod p$, we must have $\sum_{i=1}^n c_i a_i^{j_0}=0$, and thus the generalized Vandermonde matrix 
$$(a_i^j)_{1 \leq i \leq n, j \in \{0,1, \ldots, n-2\} \cup \{j_0\}}$$ must be singular (for otherwise, all $c_i$ must be 0). For such a $j_0$, we can use Jacobi's bialternant formula to compute the determinant of the above generalized Vandermonde matrix and it follows that 
$$0=s_{(j_0-(n-1),0, 0, \ldots, 0)}(a_1,a_2, \ldots, a_n),
$$
where $s$ is the Schur polynomial, and thus
$$
h_{j_0-(n-1)}(a_1,a_2, \ldots, a_n)=\sum_{1 \leq i_1 \leq i_2 \leq \cdots \leq i_{j_0-(n-1)}\leq n} a_{i_1}a_{i_2} \cdots a_{i_{j_0-(n-1)}}=0
$$ 
where $h_{j_0-(n-1)}$ is the complete homogeneous symmetric polynomial of degree $j_0-(n-1)$ (see for example \cite[Chapter 4]{S03}). This observation, together with some basic properties of symmetric polynomials, allows us to predict the algebraic structure of $A$, as well as eliminate certain sizes of $A$.
\end{rem}

\begin{ex}
Let $q=p^2$, where $p$ is a sufficiently large prime. Assume that $A, B \subset \F_q$ such that $A+B=S_2$ with $|A|, |B| \geq 2$ and $|A||B|>|S_2|$. Then we can use Proposition~\ref{prop:structure} to obtain some nontrivial information of $|A|$ and $|B|$. Without loss of generality, assume that $|A| \leq |B|$. Note that the $p$-adic expression of $\frac{q-1}{2}$ is simply $(\frac{p-1}{2}, \frac{p-1}{2})_p$. Since $\binom{|A|-1+\frac{q-1}{2}}{\frac{q-1}{2}} \equiv 0 \pmod p$, we must have $|A|\geq \frac{p+3}{2}$. On the other hand, Lemma~\ref{lem:ub} implies that $|A||B|< q=p^2$ and thus $|B| \leq 2(p-3)$. Since $\binom{|B|-1+\frac{q-1}{2}}{\frac{q-1}{2}} \equiv 0 \pmod p$, we must have $|B|=ap+b$, where $b \in \{0, \frac{p+3}{2}, \frac{p+5}{2}, \ldots, p-1\}$. Thus, there are two possibilities:
\begin{itemize}
    \item $\frac{3p+3}{2}\leq |B| \leq 2p-6$, and $\frac{p+3}{2}\leq |A|<\frac{2p}{3}$.
    \item $\sqrt{\frac{p^2-1}{2}}<|B|\leq p$, and $\frac{p+3}{2} \leq |A| \leq |B|$.
\end{itemize}
In particular, these estimates improves the best-known estimate that $\min \{|A|, |B|\} \geq (2+o(1)) \frac{p \log 2}{8 \log p}$ \cite[Lemma 5.1]{S13} under the assumption that $|A||B|>|S_2|$. One can also compare our bounds with the best-known estimates on $|A|$ and $|B|$ over prime fields by Chen and Xi \cite{CX22}.
\end{ex}

\section{Proof of Theorem~\ref{thm:A+A}}

In this section, we study additive decompositions of the special form $A+A$.

\begin{proof}[Proof of Theorem~\ref{thm:A+A}]\label{sec:A+A}
Let $q=p^n$, and let $(d,q)$ be a good pair. For the sake of contradiction, assume that $A+A=S_d$ for some $A \subset \F_q$. Then Lemma~\ref{lem:ub} implies that $|A|<\sqrt{q}=p^{n/2}$. On the other hand, since $a+a'=a'+a$ for each $a,a' \in A$, it follows that
\begin{equation}\label{eq/2}
 \frac{q-1}{d}=|S_d|=|A+A|\leq \frac{|A|^2+|A|}{2}.    
\end{equation}

If $\frac{q-1}{d} \leq \frac{2p}{3}$ (in particular, if $q=p$ is a prime), then by Corollary~\ref{cor:a+b}, $|A|^2=\frac{q-1}{d}$. Thus, inequality~\eqref{eq/2} implies that $|A|=1$. It follows that $|S_d|=|A+A|=1$ and thus $d=q-1$, violating our assumption.

Next, we assume that $q$ is a proper prime power. Let $k$ be the unique integer such that $p^k \leq |A|-1<p^{k+1}$. Since $|A|<p^{n/2}$, it follows that $k<n/2$. Thus we can write $$|A|-1=(c_k, c_{k-1}, \ldots, c_1,c_0)_p$$ in base-$p$, that is, $|A|-1=\sum_{i=0}^k c_ip^i$ with $0 \leq c_i \leq p-1$ for each $0 \leq i \leq k$ and $c_k \geq 1$. Also write $$\frac{q-1}{d}=(e_n, e_{n-1}, \ldots, e_1,e_0)_p$$ in base-$p$; note that it is possible that $e_n=0$.

The key idea is to apply Corollary~\ref{cor:cora+b}. To do so, we take a subset $A'$ of $A$ such that 
\begin{equation}\label{binomial}
 \binom{|A'|-1+\frac{q-1}{d}}{\frac{q-1}{d}}\not \equiv 0 \pmod p,   
\end{equation}
then Corollary~\ref{cor:cora+b} implies that $|A||A'|\leq \frac{q-1}{d}$. To derive the contradiction, in view of inequality~\eqref{eq/2}, it suffices to show that we can pick $A'$ so that the inequality $|A'|> \frac{|A|+1}{2}$ and equation~\eqref{binomial} both hold. Also note that if $|A'|=\frac{|A|+1}{2}$, then we must have $|A||A'|=\frac{q-1}{d}$. 

Next, we consider the other three sufficient conditions for a good pair $(d,q)$ separately:

(1) Assume that $e_j \leq \frac{p-1}{2}$ for each $j<n/2$. Since $k<n/2$, we have $e_j \leq \frac{p-1}{2}$ for each $0 \leq j \leq k$. We construct a desired $A'$ according to the following two cases:
\begin{itemize}
    \item 
If $c_k\leq \frac{p-1}{2}$, then we can pick $A'$ such that $|A'|-1=c_kp^k$ so that $$c_k+e_k\leq \frac{p-1}{2}+\frac{p-1}{2} \leq p-1.$$
Thus, Lucas' theorem implies that
$$
\binom{|A'|-1+\frac{q-1}{d}}{\frac{q-1}{d}}\equiv \prod_{j=k+1}^n \binom{e_j}{e_j} \cdot \binom{c_k+\frac{p-1}{d}}{\frac{p-1}{d}} \cdot \prod_{j=0}^{k-1} \binom{e_j}{e_j}\not \equiv 0 \pmod p,
$$
In this case, $|A|\leq (c_k+1)p^k$ and thus
$$
|A'|=c_kp^k+1= \frac{2c_kp^k+2}{2}>\frac{(c_k+1)p^k+1}{2}\geq \frac{|A|+1}{2}
$$
by the assumption that $c_k \geq 1$.

\item
If $c_k>\frac{p-1}{2}$, then we can pick $A'$ such that $|A'|-1=\frac{p^{k+1}-1}{2}$ so that 
$$
\binom{|A'|-1+\frac{q-1}{d}}{\frac{q-1}{d}}\equiv \prod_{j=k+1}^n \binom{e_j}{e_j} \cdot \prod_{j=0}^{k} \binom{\frac{p-1}{2}+e_j}{e_j}\not \equiv 0 \pmod p,
$$
where we used Lucas' theorem. Note that in this case, by the assumption, $|A|-1<p^{k+1}$ so $|A|\leq p^{k+1}$. It follows that $|A'|=\frac{p^{k+1}+1}{2} \geq \frac{|A|+1}{2}$. Thus, it suffices to rule out the case of equality. However, if $|A'|=\frac{|A|+1}{2}$, then we must have $|A|=p^{k+1}$ and $|A||A'|=\frac{q-1}{2}$, implying that $q-1=p^{2k+2}+p^{k+1}$, which is clearly impossible by considering modulo $p$ on both sides of the equation.
\end{itemize}

(2) Assume that $n=2r+1$ with $r \geq 1$, $d\leq 2p-2$, and $e_r\leq p-1-\lceil \sqrt{p} \rceil/2$. Since $d\leq 2p-2$, inequality~\eqref{eq/2} implies that
$$
p^{2r}+p^r<\sum_{i=0}^{2r} p^i=\frac{p^{2r+1}-1}{p-1} \leq \frac{2(q-1)}{d}\leq |A|^2+|A| \implies p^r<|A|.
$$
Thus, $p^r \leq |A|-1<\sqrt{q}-1=\sqrt{p} \cdot p^r-1$. It follows that $k=r$ and $1\leq c_r\leq \lfloor \sqrt{p} \rfloor$. 
\begin{itemize}
    \item If $c_r+e_r\leq p-1$, then we can pick $A'$ such that $|A'|-1=c_rp^r$. Similar to the analysis in case (1), we can show that the inequality $|A'|> \frac{|A|+1}{2}$ and equation~\eqref{binomial} both hold. 
    \item If $c_r+e_r \geq p$, then we can pick $A'$ such that $|A'|-1=(p-1-e_r)p^r$. We can verify that equation~\eqref{binomial} holds in a similar way using    Lucas' theorem. By the assumption, $e_r\leq p-1-\lceil \sqrt{p} \rceil/2$. It follows that $p-1-e_r \geq \lceil \sqrt{p} \rceil/2$ and thus
    $$|A'| \geq (\lceil \sqrt{p} \rceil/2)p^r+1>\frac{p^{r+1/2}}{2}+1=\frac{\sqrt{q}+2}{2}>\frac{|A|+1}{2}.
    $$

\end{itemize}

(3) Assume that $n=2r$, $d \leq 2p^2$, and $e_{r-1} \leq (p-3)/2$. 
Similar to the analysis in case (2), we can show that $p^{r-1} \leq |A|-1<\sqrt{q}=p \cdot p^{r-1}$ so that $k=r-1$. Let $c_{r-1}'=\min \{c_{r-1}, \frac{p+1}{2}\}$ so that $e_{r-1}+c_{r-1}'\leq p-1$. Take a subset $A'$ of $A$ with $|A'|-1=c_{r-1}' \cdot p^{r-1}$. Similar to the analysis in cases (1) and (2), we can verify that the inequality $|A'|> \frac{|A|+1}{2}$ and equation~\eqref{binomial} both hold. 
\end{proof}

The following corollary is a special case of Corollary~\ref{cor:A+A}.

\begin{cor}
Let $q=p^n$ with $q \equiv 1 \pmod d$ and $d \geq 2$. Assume that $p \equiv 1 \pmod d$ holds, or $p \equiv -1 \pmod d$ and $4 \mid n$ both hold. Then there is no additive decomposition $S_d=A+A$ for any $A \subset \F_q$. 
\end{cor}

The next example shows that it is necessary to assume that $d \geq 2$ in Theorem~\ref{thm:A+A}, and have some additional assumptions when $q$ is a proper prime power.

\begin{ex}\label{ex:ex1}
Let $p\geq 7$ be an odd prime, and $n$ be a positive integer. We claim that $\F_{p^n}^*$ can be decomposed as $A+A$. Indeed, since $\F_{p^n}$ is isomorphic to the $n$-dimensional space over $\F_p$, it suffices to show that $\F_{p}^n \setminus \{\mathbf{0}\}$ can be decomposed as $A+A$, where $A \subset \F_{p}^n$. We can take 
$$
A=\bigg(\bigg\{0,1,2, \ldots, \frac{p-3}{2}\bigg\} \cup \bigg\{\frac{p+1}{2}\bigg\}\bigg)^n \setminus \{\mathbf{0}\}
$$
to achieve this purpose. Also note that $|A|=(\frac{p+1}{2})^n-1$ and thus $|A|^2/p^n$ could be arbitrarily large.

Let $q=p^n$, with $p \geq 7$ and $n \geq 2$. Let $k$ be a proper divisor of $n$ and let $d=\frac{q-1}{p^k-1}$. Then $S_d=\F_{p^k}^*$ can be decomposed as $A+A$ by the above argument. This example shows that when $q$ is a proper prime power, one needs to impose some conditions on the pair $(d,q)$ so that Theorem~\ref{thm:A+A} holds. It also shows that the condition on the binomial coefficient in Theorem~\ref{thm:main} cannot be dropped, and shows that Corollary~\ref{cor:a+b} does not extend to all pairs $(d,q)$. 
\end{ex}

We end the section by comparing our new results with existing results.

\begin{rem}\label{rmk:A+A}
Shkredov's proof \cite{S14} that $A+A \neq S_2$ for any $A \subset \F_p$ with $p\geq 5$ is elegant and remarkable. In fact, the identical proof shows that $A+A \neq S_2$ for any $A \subset \F_q$, where $q \geq 5$ is an odd prime power. Here we sketch his proof. From inequality~\eqref{eq:double}, we can deduce that $|A|\leq \sqrt{q} (1-\frac{|A|}{q})$, equivalently, $q+\frac{|A|^2}{q} \geq |A|^2+2|A|$. On the other hand, $|A+A|=|S_2|$ implies that $|A|^2+|A| \geq q-1$ (see inequality~\eqref{eq/2}). It is easy to deduce a contradiction by comparing the above two inequalities. Unfortunately, it seems his proof does not extend to $S_d$ for $d \geq 3$. Nevertheless, Shkredov also showed a similar result for the restricted sumset $A \hat{+} A = \{ a + b: a,b \in A, \, a \neq b \}$, for which our techniques do not seem to apply. 
\end{rem}

\begin{rem}
In fact, in his paper, Shkredov~\cite{S14} established a much stronger result, namely, $A+A$ is impossible to be too close to $S_2$. We remark that under the same assumption on the pair $(d,q)$ in Corollary~\ref{cor:a+b}, Corollary~\ref{cor:a+b} implies that: if $A+A \subset S_d$, then $|A+A|\leq (\frac{1}{2}+o(1))|S_d|$. In particular, this improves and generalizes Shkredov's result \cite[Theorem 3.2]{S14}. The proof of Theorem~\ref{thm:A+A} could be also adapted to estimate the maximum size of $A+A$ (or the maximum size of $A$) provided that $A+A \subset S_d$.
\end{rem}

\begin{rem}
Wu and She \cite{WS23} recently studied the additive decomposition of $S_3$. Let $p \equiv 1 \pmod 3$. They showed that if $A+B=S_3$ holds for some $A, B \subset \F_p$ with $|A|=|B|$, then $\sqrt{\frac{p-1}{3}} \leq |A| \leq \sqrt{p}$ \cite[Theorem 1.3]{WS23}. Note that Corollary~\ref{cor:a+b} shows that we must have $|A|=|B|=\sqrt{\frac{p-1}{3}}$. They also had an improved lower bound on $|A|$ if $A=B$, however, this case has been ruled out by Corollary~\ref{cor:A+A}.
\end{rem}

\begin{rem}
The possibility of writing $S_d \cup \{0\}$ as the difference set $A-A$ have also been studied extensively; see for example \cite{S16, LS17, S18, HP, Y21,Y22, Yip1}. As mentioned in Remark~\ref{rem:Paley}, this problem is also closely related to the clique number of generalized Paley graphs. More generally, this problem is related to the Paley graph conjecture (see for example \cite[Section 2.2]{Y22}), which is widely open.

Murphy, Petridis, Roche-Newton, Rudnev, and Shkredov \cite[Theorem 16]{MPRRS19} showed that multiplicative subgroups of size less than $p^{6/7-o(1)}$ cannot be represented in the form $(A-A) \setminus \{0\}$ for any $A \subset \F_p$. Unfortunately, it seems the proof techniques in Theorem~\ref{thm:A+A} fail to extend to $A-A$. Nevertheless, the proof of Corollary~\ref{cor:a+b} can be modified slightly to show that under the same condition on the pair $(d,q)$, if $A \subset \F_q$ such that $A-A=S_d \cup \{0\}$, then $|A|^2-|A|=|S_d|$. In particular, this leads to new families of pairs $(d,q)$ such that $A-A \neq S_d \cup \{0\}$ for any $A \subset \F_q$, which is beyond \cite{S16} and \cite[Corollary 1.6]{HP}.
\end{rem}

\section{No nontrivial ternary decomposition}\label{sec:A+B+C}

In this section, we establish various sufficient conditions on the pair $(d,q)$ so that $S_d \subset \F_q$ admits no nontrivial ternary decomposition.

\begin{proof}[Proof of Theorem~\ref{thm:prime3}]
Assume that there exists a prime $p$ and a multiplicative subgroup $G$ of $\F_p$, such that $G$ admits a nontrivial additive decomposition: $G=A+B+C$ with $|A|,|B|, |C| \geq 2$. Then we can write $G$ in three different ways: $$A+(B+C), B+(C+A), C+(A+B).$$ Applying Theorem~\ref{thm:lball} to each sumset, we have $$|A|, |B|, |C| \gg |G|^{1/2+o(1)}.$$ We can also apply Corollary~\ref{cor:a+b} to obtain that 
$$
|A+B||C|,|B+C||A|,|C+A||B|\ll |G|.
$$
Therefore, from Lemma~\ref{lem:addcomb} and the fact $|A+B+C|=|G|$, we have
$$
|G|^2 |A||B||C|\ll |A+B+C|^2 |A||B||C| \leq (|A+B||C|)(|B+C||A|)(|C+A||B|) \ll |G|^3.
$$
It follows that 
$$
|G|^{3/2+o(1)} \ll |A||B||C| \ll |G|,
$$
that is, $|G|\ll 1$, where the implicit constant is absolute. This completes the proof of the theorem.
\end{proof}

\begin{rem}
When $d$ is fixed, and $p \equiv 1 \pmod d$ is sufficiently large, there is an alternative way to show that there is no nontrivial ternary decomposition of $S_d$. Suppose $A+B+C=S_d$, where $A, B, C \subset \F_p$ and $|A|, |B|, |C| \geq 2$. Then Lemma~\ref{lem:maxmin} implies that $|A|, |B|, |C| \gg \sqrt{p}$. However, when $p$ is sufficiently large, this would violate a ternary character sum estimate proved by Hanson \cite[Theorem 1]{H17} that 
$$
\sum_{a \in A, b \in B, c \in C} \chi(a+b+c)=o(|A||B||C|),
$$
where $\chi$ is a multiplicative character with order $d$. 
\end{rem}

Next, we modify the proof of Theorem~\ref{thm:prime3} to establish Theorem~\ref{thm:dq}.

\begin{proof}[Proof of Theorem~\ref{thm:dq}]
The proof is similar to that of Theorem~\ref{thm:prime3}. Let $\epsilon$ be fixed. Assume that $(d,q)$ is a pair with $d \leq q^{1/10-\epsilon}$, such that $S_d$ can be decomposed into $S_d=A+B+C$ for some $A,B,C \subset \F_q$ with $|A|,|B|, |C| \geq 2$. Applying Lemma~\ref{lem:maxmin} and Lemma~\ref{lem:ub} to each of the sumsets $A+(B+C), B+(C+A), C+(A+B)$, we have
$$
|A|, |B|, |C| \gg \frac{\sqrt{q}}{d}, \quad  |A||B+C|,|B+C||A|, |C+A||B| <q.
$$
Therefore, the same argument as in the proof of Theorem~\ref{thm:prime3} gives the estimate
$$
\frac{q^{3.5}}{d^5} \ll |A+B+C|^2 |A||B||C| \leq (|A+B||C|)(|B+C||A|)(|C+A||B|) \ll q^3.
$$
It follows that $d \gg q^{1/10}$. Together with the assumption $d\leq q^{1/10-\epsilon}$, we conclude that $q\ll 1$, where the implicit constant depends only on $\epsilon$. This completes the proof of the theorem.
\end{proof}

The next example provides a counterexample when the conditions in Theorem~\ref{thm:prime3} and Theorem~\ref{thm:dq} are weakened.

\begin{ex}\label{ex:ex2}
Let $p\geq 5$ be a prime, and $n$ be a positive integer. We claim that $\F_{p^n}^*$ can be decomposed as $A+B+C$ nontrivially. Indeed, since $\F_{p^n}$ is isomorphic to the $n$-dimensional space over $\F_p$, it suffices to show that $\F_{p}^n \setminus \{\mathbf{0}\}$ can be decomposed as $A+B+C$, where $A,B,C \subset \F_{p}^n$. We can take 
$$
A=\{0,1\}^n, \quad B=\{0,1\}^n, \quad \text{and } C=\{0,1,2, r+3, r+6, \ldots, p-3\}^n \setminus \{\mathbf{0}\}
$$
to achieve this purpose, where $p \equiv r \pmod 3$ with $r \in \{1,2\}$.  

Let $q=p^n$, with $p \geq 5$ and $n \geq 2$. Let $k$ be a proper divisor of $n$ and let $d=\frac{q-1}{p^k-1}$. Then $S_d=\F_{p^k}^*$ can be decomposed as $A+B+C$ nontrivially by the above argument. Note that $|S_d|$ could be arbitrarily large, which shows that Theorem~\ref{thm:prime3} fails to extend to all finite fields. It also shows that in the statement of Theorem~\ref{thm:dq}, it is necessary to impose the condition that $q$ is sufficiently large compared to $d$; in particular, the exponent $\frac{1}{10}$ cannot be replaced by any constant which is greater than $\frac{1}{2}$. It is interesting to explore if the exponent $\frac{1}{10}$ could be improved. Under extra assumptions, we improve the exponent to roughly $\frac{1}{4}$ in Theorem~\ref{thm:deltagood}.
\end{ex}

Next, we combine the ideas used in the previous discussions to prove Theorem~\ref{thm:deltagood}.

\begin{proof}[Proof of Theorem~\ref{thm:deltagood}]
Let $\epsilon$ be fixed. Let $M$ be the implicit constant from inequality~\eqref{eq:max}; note $M$ only depends on $\epsilon$. Let $P=M^2+1$.

Let $q=p^n$ with $p>P$. Assume that $(d,q)$ is a $\delta$-good pair with $\delta \in (0,1)$ and $d \leq q^{1/4-\epsilon}\delta^{3/2}$, such that $S_d$ can be decomposed into $S_d=A+B+C$ for some $A,B,C \subset \F_q$ with $|A|,|B|, |C| \geq 2$. Then we can write $S_d$ in three different ways: $A+(B+C), B+(C+A), C+(A+B)$. Applying Lemma~\ref{lem:maxmin} to each of these sumsets, we have 
\begin{equation}\label{eq:lb}
|A|, |B|, |C| \gg \frac{\sqrt{q}}{d},    
\end{equation}
and 
$$
|A|, |B|, |C| \leq M\sqrt{q}<p^{(n+1)/2}.
$$

Let $k$ be the unique integer such that $p^k \leq |A|-1<p^{k+1}$. By the assumption, $k<(n+1)/2$. Then we can write $$|A|-1=(c_k, c_{k-1}, \ldots, c_1,c_0)_p$$ in base-$p$, where $c_k\geq 1$. Also write $$\frac{q-1}{d}=(e_n, e_{n-1}, \ldots, e_1,e_0)_p$$ in base-$p$, where $e_n$ is possibly $0$.

Since $(d,q)$ is a $\delta$-good pair, we have $e_k\leq \lfloor(1-\delta)(p-1)\rfloor$. Let $A'$ be a subset of $A$ such that $|A'|-1=c_k' p^k$, where $c_k'=\min \{c_k, \lceil \delta (p-1) \rceil\}$ so that 
$$c_k'+e_k\leq \lceil \delta (p-1) \rceil +\lfloor(1-\delta)(p-1)\rfloor= p-1.$$
Then Lucas' theorem implies that
$$
\binom{|A'|-1+\frac{q-1}{d}}{\frac{q-1}{d}}\equiv \prod_{j=k+1}^n \binom{e_j}{e_j} \cdot \binom{c_k'+e_k}{e_k} \cdot \prod_{j=0}^{k-1} \binom{e_j}{e_j}\not \equiv 0 \pmod p.
$$
Similar to the proof of Theorem~\ref{thm:A+A}, it is easy to verify that $$|A'| \gg \min \{|A|/2,\delta|A|\} \gg \delta |A|$$ from our construction. Since $A' +(B+C) \subset A+B+C=S_d$, Corollary~\ref{cor:a+b} implies that 
$$
|A||B+C| \ll \frac{|A'||B+C|}{\delta}\ll \frac{q}{d\delta}.
$$

Using a similar argument, we can also show that
$$
|A+B||C|\ll \frac{q}{d\delta}, \quad |C+A||B|\ll \frac{q}{d\delta}.
$$
Therefore, from Lemma~\ref{lem:addcomb} and the fact $|A+B+C|=|S_d|$, we have
$$
\frac{q^2}{d^2} |A||B||C|\ll |A+B+C|^2 |A||B||C| \leq (|A+B||C|)(|B+C||A|)(|C+A||B|) \ll \frac{q^3}{d^3\delta^3}.
$$
Combining inequality~\eqref{eq:lb}, we have 
$$
\frac{q^{3/2}}{d^3} \ll |A||B||C| \ll \frac{q}{d \delta^3},
$$
that is, $d \gg q^{1/4} \delta^{3/2}$. Together with the assumption $d\leq q^{1/4-\epsilon} \delta^{3/2}$, we conclude that $q\ll 1$, where the implicit constant depends only on $\epsilon$. This completes the proof of the theorem.
\end{proof}

Finally, we deduce a simple corollary of Theorem~\ref{thm:deltagood}.

\begin{cor}\label{cor:order1}
There is an absolute constant $P$, such that for pair of prime power $q=p^n$ and $d$ such that  $n \geq 5$, $p>P$, and $d \mid (p-1)$, we have $S_d \neq A+B+C$ for any $A,B,C \subset \F_q$ with $|A|,|B|,|C| \geq 2$.  
\end{cor}

\begin{proof}
Take $\epsilon=\frac{1}{20}$ so that $d \leq p-1<p\leq q^{1/5}=q^{1/4-\epsilon}$. Note that all digits in the base-$p$ representation of $\frac{q-1}{d}$ are equal to $\frac{p-1}{d}\leq \frac{p-1}{2}$. So the pair $(d,q)$ is $\delta$-good with $\delta \geq \frac{1}{2}$.  The corollary follows from Theorem~\ref{thm:deltagood}.
\end{proof} 

\section*{Acknowledgement}
The author thanks Seoyoung Kim, Greg Martin, J\'ozsef Solymosi, and Semin Yoo for helpful discussions. The author is also grateful to anonymous referees for their valuable comments and suggestions.

\bibliographystyle{abbrv}
\bibliography{main}

\end{document}